\documentclass[draft]{article}
\usepackage{amssymb,amsmath,amsthm,latexsym}

\newtheorem{theorem}{Theorem}
\newtheorem{lemma}{Lemma}
\newtheorem{corollary}{Corollary}
\newtheorem{example}{Example}

\newcommand{\R}{{\mathbb{R}}}

\newcommand{\Bd}{\mathrm{bd\,}}
\newcommand{\Cl}{\mathrm{cl\,}}
\newcommand{\Conv}{\mathrm{conv\,}}
\newcommand{\Dim}{\mathrm{dim\,}}
\newcommand{\Int}{\mathrm{int\,}}
\newcommand{\Lin}{\mathrm{lin\,}}

\newcommand{\Rec}{\mathrm{rec\,}}

\begin{document}

\begin{center} {\Large \textbf{Convex sets with homothetic
projections}}

\bigskip

\textbf{Valeriu~Soltan}

\medskip

Department of Mathematical Sciences, George Mason University

4400 University Drive, Fairfax, VA 22030, USA

vsoltan@gmu.edu

\end{center}

\smallskip

\begin{quote}
\noindent \textbf{Abstract.} Nonempty sets $X_1$ and $X_2$ in the
Euclidean space $\R^n$ are called \textit{homothetic} provided
$X_1 = z + \lambda X_2$ for a suitable point $z \in \R^n$ and a
scalar $\lambda \ne 0$, not necessarily positive. Extending
results of S\"uss and Hadwiger (proved by them for the case of
convex bodies and positive $\lambda$), we show that compact
(respectively, closed) convex sets $K_1$ and $K_2$ in $\R^n$ are
homothetic provided for any given integer $m$, $2 \le m \le n - 1$
(respectively, $3 \le m \le n - 1$), the orthogonal projections of
$K_1$ and $K_2$ on every $m$-dimensional plane of $\R^n$ are
homothetic, where homothety ratio may depend on the projection
plane. The proof uses a refined version of Straszewicz's theorem
on exposed points of compact convex sets.

\smallskip


\smallskip

\end{quote}

\vskip0.5cm

\section{Introduction and main results}

\noindent Let us recall that nonempty sets $X_1$ and $X_2$ in the
Euclidean space $\R^n$ are \textit{homothetic} provided $X_1 = z +
\lambda X_2$ for a suitable point $z \in \R^n$ and a scalar
$\lambda \ne 0$ (called homothety ratio); furthermore, $X_1$ and
$X_2$ are called \textit{positively} homothetic (respectively,
\textit{negatively} homothetic) provided $\lambda > 0$
(respectively, $\lambda < 0$). We remark that in convex geometry
homothety usually means positive homothety, also called direct
homothety. In a standard way, a convex body in $\R^n$ is a compact
convex set with nonempty interior.

S\"uss~\cite{sus29,sus32} proved that a pair of convex bodies in
$\R^n$, $n \ge 3$, are positively homothetic if and only if the
orthogonal projections of these bodies on every hyperplane are
positively homothetic, where homothety ratio may depend on the
projection hyperplane (the proof of this statement is given for $n
= 3$ with the remark in \cite[p.\,49]{sus32} that the extension to
higher dimensions is routine). Following a series of intermediate
results (see~\cite{sol05b} for additional references),
Hadwiger~\cite{had63} showed that convex bodies $K_1$ and $K_2$ in
$\R^n$ are positively homothetic if and only if there is an
integer $m$, $2 \le m \le n - 1$, such that the orthogonal
projections of $K_1$ and $K_2$ on each $m$-dimensional plane are
positively homothetic (see also Rogers~\cite{rog65} for the case
$m = 2$).

The question whether the statements of S\"uss and Hadwiger hold
for larger families of geometric transformations in $\R^n$, like
similarities, was posed by Nakajima~\cite[p.\,169]{nak30} for $n =
3$ and independently by Petty and McKinney \cite{p-m87} and
Golubyatnikov~\cite{gol88}. Gardner and Vol\v ci\v c \cite{g-v94}
showed the existence of a pair of centered and coaxial convex
bodies of revolution in $\R^n$ whose orthogonal projections on
every 2-dimensional plane are similar, but which are not
themselves even affinely equivalent. On the other hand,
Golubyatnikov~\cite{gol88,gol89} proved that compact convex sets
$K_1$ and $K_2$ in $\R^n$ are homothetic (positively or
negatively) provided their projections on every 2-dimensional
plane are similar and have no rotation symmetries.

Our first theorem shows that the family of positive homotheties in
S\"uss's and Hadwiger's statements can be extended to all
homotheties in $\R^n$.

\begin{theorem} \label{th1}
Given compact $($respectively, closed$)$ convex sets $K_1$ and
$K_2$ in $\R^n$ and an integer $m$, $2 \le m \le n - 1$
$($respectively,  $3 \le m \le n - 1)$, the following conditions
are equivalent:
\begin{itemize}

\item[$1)$] $K_1$ and $K_2$ are homothetic,

\item[$2)$] the orthogonal projections of $K_1$ and $K_2$ on every
$m$-dimensional plane of $\R^n$ are homothetic, where homothety
ratio may depend on the projection plane.

\end{itemize}

\end{theorem}

The following example shows that the inequality $m \ge 3$ in
Theorem~\ref{th1} is sharp for the case of unbounded convex sets.

\begin{example} \label{ex1}
\textnormal{Let $K_1$ and $K_2$ be solid paraboloids in $\R^3$,
given, respectively, by
\[
K_1 = \{(x,y,z) \mid x^2 + y^2 \le z\} \quad \mathrm{and} \quad
K_2 = \{(x,y,z) \mid 2x^2 + y^2 \le z\}.
\]
Obviously, $K_1$ and $K_2$ are not homothetic. At the same time,
their orthogonal projections $\pi_L(K_1)$ and $\pi_L(K_2)$ on
every 2-dimensional plane $L \subset \R^3$ are positively
homothetic. Indeed, if $L = \{(x,y,z) \mid z = \text{const} \}$,
then $\pi_L(K_1) = \pi_L(K_2) = L$. For any other 2-dimensional
plane $L$ in $\R^3$, the sets $\pi_L(K_1)$ and $\pi_L(K_2)$ are
closed convex sets bounded by parabolas whose axes of symmetry are
parallel to the orthogonal projection of the $z$-axis on $L$.
Since any two parabolas in the plane with parallel axes of
symmetry are homothetic, the sets $\pi_L(K_1)$ and $\pi_L(K_2)$
also are positively homothetic.}
\end{example}

In view of this example, it would be interesting to describe the
pairs of closed convex sets $K_1$ and $K_2$ in $\R^n$ such that
the orthogonal projections of $K_1$ and $K_2$ on every
$2$-dimensional plane of $\R^n$ are homothetic. The following
corollary slightly refines Theorem~\ref{th1}.

\begin{corollary} \label{cor1}
Given compact $($respectively, closed$)$ convex sets $K_1$ and
$K_2$ in $\R^n$, integers $r$ and $m$ such that $0 \le r \le m - 2
\le n - 3$ $($respectively, $0 \le r \le m - 3 \le n - 4)$, and a
subspace $S \subset \R^n$ of dimension $r$, the following
conditions are equivalent:
\begin{itemize}

\item[$1)$] $K_1$ and $K_2$ are homothetic,

\item[$2)$] the orthogonal projections of $K_1$ and $K_2$ on every
$m$-dimensional plane of $\R^n$ that contains $S$ are homothetic,
where homothety ratio may depend on the projection plane.

\end{itemize}

\end{corollary}

We observe that the proof of Theorem~\ref{th1} cannot be routinely
reduced to that of S\"uss and Hadwiger by using compactness and
continuity arguments. Indeed, if orthogonal projections
$\pi_L(K_1)$ and $\pi_L(K_2)$ of the convex sets $K_1$ and $K_2$
on a plane $L \subset \R^n$ are homothetic and
\[
\pi_L (K_1) = z_L + \lambda_L \pi_L (K_2),
\]
then $z_L$ and $\lambda_L$ (but not the absolute value of
$\lambda_L$) may loose their continuity as functions of $L$ when
both $\pi_L(K_1)$ and $\pi_L(K_2)$ become centrally symmetric. To
avoid the consideration of centrally symmetric projections, our
proof of Theorem~\ref{th1} uses a refined version of Straszewicz's
theorem on exposed points of a compact convex set (see
Theorem~\ref{th2} below).

Let us recall that a point $x$ of a closed convex set $K \subset
\R^n$ is called \textit{exposed} provided there is a hyperplane $H
\subset \R^n$ supporting $K$ such that $H \cap K = \{ x \}$.
Straszewicz's theorem states that any compact convex set in $\R^n$
is the closed convex hull of its exposed points
(see~\cite{str35}). Klee~\cite{kle57} proved that a line-free
closed convex set $K \subset \R^n$ is the closed convex hull of
its exposed points and exposed halflines (a set is called
\textit{line-free} if it contains no lines).

Points $x$ and $z$ of a compact convex set $K \subset \R^n$ are
called (affinely) \textit{antipodal} provided there are distinct
parallel hyperplanes $H$ and $G$ both supporting $K$ such that $x
\in H \cap K$ and $z \in G \cap K$ (see,
e.\,g.,~\cite{m-s05,sol05a} for various antipodality properties of
convex and finite sets in $\R^n$). Furthermore, the points $x$ and
$z$ are called \textit{antipodally exposed} (and the chord $[x,z]$
is called an \textit{exposed diameter} of $K$) provided the
parallel hyperplanes $H$ and $G$ can be chosen such that $H \cap K
= \{ x \}$ and $G \cap K = \{ z \}$ (see~\cite{n-s94,s-n93}).
Clearly, a compact convex set may have exposed points which are
not antipodally exposed (like the point $x$ in
Figure~\ref{fig-1}).

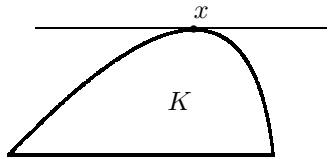
\begin{figure}[h]
\begin{minipage}{11.7cm}

\begin{center}
%\fbox{
\begin{picture}(200,55)
 \put(60,48){\line(1,0){110}}
\thicklines \qbezier(50,0)(140,95)(150,00)
\put(50,0){\line(1,0){100}} \put(120,48){\circle{2}}
\put(120,52){$x$} \put(110,17){$K$}
\end{picture}
%}
\end{center}

\end{minipage}
\caption{An exposed point which is not antipodally exposed.}
\label{fig-1}
\end{figure}

\begin{theorem} \label{th2}
Any compact convex set $K \subset \R^n$ distinct from a singleton
is the closed convex hull of its antipodally exposed points.
\end{theorem}

In what follows, we will need the following lemma.

\begin{lemma} \label{lemma-a1}
No two distinct exposed diameters of a compact convex set $K
\subset \R^n$ are parallel.
\end{lemma}

\begin{proof}
Assume for a moment that $K$ has a pair of distinct parallel
exposed diameters, say $[x_1, z_1]$ and $[x_2, z_2]$. We may
suppose that $x_1 - z_1$ and $x_2 - z_2$ have the same direction
and $\| x_1 - z_1 \| \le \| x_2 - z_2 \|$. Denote by $H$ and $G$
distinct parallel hyperplanes both supporting $K$ such that $H
\cap K = \{x_1\}$ and $G \cap K = \{z_1\}$. Let $[x_2', z_2']$ be
the intersection of the line $(x_2, z_2)$ and the closed slab
between $G$ and $H$. Clearly, $\| x_2' - z_2' \| = \| x_1 - z_1
\|$. Since $[x_2, z_2] \subset K$, we conclude that $[x_2, z_2]
\subset [x_2', z_2']$. Then $[x_2, z_2] = [x_2', z_2']$ because of
$\| x_1 - z_1 \| \le \| x_2' - z_2' \|$. Hence $x_2 \in H$ and
$z_2 \in G$. Due to $H \cap K = \{x_1\}$ and $G \cap K = \{z_1\}$,
we obtain $[x_1, z_1] = [x_2, z_2]$, in contradiction with the
choice of these diameters.
\end{proof}

We conclude this section with necessary definitions, notions, and
statements (see, e.\,g., \cite{web94} for general references).  In
a standard way, $\Bd K$, $\Cl K$, and $\Int K$ denote the
boundary, the closure, and the interior of a convex set $K \subset
\R^n$; the \textit{recession cone} of $K$ is defined by
\[
\Rec K = \{ y \in \R^n \mid x + \alpha y \in K \ \hbox{for all} \
x \in K \ \mathrm{and} \ \alpha \ge 0\}.
\]
It is well-known that $\Rec K \ne \{o\}$ if and only if $K$ is
unbounded. The \textit{linearity spaces} $\Lin K$ of $K$ is given
by $\Lin K = (\Rec K) \cap (-\Rec K)$, and $K$ can be expressed as
the direct sum $K = \Lin K \oplus (K \cap M)$, where the subspace
$M$ is the orthogonal complement of $\Lin K$ and $K \cap M$ is a
line-free closed convex set

We say that a closed halfspace $P$ of $\R^n$ supports $K$ provided
the boundary hyperplane of $P$ supports $K$ and the interior of
$P$ is disjoint from $K$. If the halfspace $P$ is given by $P =
\{x \in \R^n \mid x \! \cdot \! f \ge \alpha\}$ where $f$ is a
unit vector and $\alpha$ is a scalar, then $f$ is called the
\emph{inward unit normal} of $P$. Closed halfspaces $S$ and $T$ in
$\R^n$ are called \textit{opposite} provided they can be written
as
\begin{equation} \label{eqn1}
S = \{x \in \R^n \mid x \! \cdot \! g \ge \alpha \} \quad
\mathrm{and} \quad T = \{x \in \R^n \mid x \! \cdot \! g \le \beta
\}
\end{equation}
for a suitable unit vector $g \in \R^n$ and scalars $\alpha \ge
\beta$. Clearly, the boundary hyperplanes of opposite halfspaces
are parallel. A \textit{plane} in $\R^n$ is a set of the form $F =
z + L$, where $z \in \R^n$ and $L$ is a subspace of $\R^n$. For
any plane $L \subset \R^n$, we denote by $\pi_L(X)$ the orthogonal
projection of a set $X \subset \R^n$ on $L$. To distinguish
similarly looking elements, we write $0$ for the real number zero,
and $o$ for the origin of $\R^n$.

\section{Proof of Theorem~\ref{th2}}

\noindent We precede the proof of Theorem~\ref{th2} by two
necessary lemmas. One might observe that an alternative proof of
Lemma~\ref{lemma-21} can use a duality argument and the fact the
set of regular points of a convex body $K \subset \R^n$ is dense
in $\Bd K$ (see also~\cite{sol06}).

\begin{lemma} \label{lemma-21}
Let $K \subset \R^n$ be a compact convex set and $f$ be a unit
vector in $\R^n$. For any $\varepsilon > 0$, there is a closed
halfspace $P \subset \R^n$ such that $K \cap P$ is a singleton and
the inward unit normal $g$ of $P$ satisfies the inequality $\| f -
g \| \le \varepsilon$.
\end{lemma}

\begin{proof}
Let $Q \subset \R^n$ be the closed halfspace with inward unit
normal $f$ that supports $K$. Denote by $H$ the boundary
hyperplane of $Q$. Choose a point $v \in H \cap K$, and let $U
\subset H$ be an $(n - 1)$-dimensional closed ball with center $v$
and radius $\delta > 0$ such that the orthogonal projection of $K$
on $H$ lies in $U$. Let $l$ be the line through $v$ in the
direction of $f$. Then $K$ entirely lies in the both-way infinite
cylinder $C$ with base $U$ and axis $l$. Choose a closed ball
$B_\rho(c)$ with center $c \in l \setminus Q$ and radius $\rho >
0$ such that $K \cup U \subset B_\rho(c)$. Furthermore, we assume
that $\rho \ge \delta \! \cdot \! \sec \gamma$ where $\gamma = 2
\arcsin (\varepsilon/2)$. If $y$ is a boundary point of
$B_\rho(c)$ that lies in $C \cap Q$ and $e_y \in \R^n$ is the unit
vector such that $y + e_y$ is the outward unit normal of
$B_\rho(c)$ at $y$, then the inequality $\rho \ge \delta \! \cdot
\! \sec \gamma$ easily implies that $\| e_y - f\| \le
\varepsilon$.

By continuity, there is a scalar $\alpha \ge 0$ such that the ball
$B = B_\rho(c) - \alpha f$ contains $K$ and the boundary of $B$
has at least one, say $x$, common point with $K$. Clearly, $x \in
C$. Denote by $P$ the closed halfspace of $\R^n$ such that $B \cap
P = \{ x \}$. By the above, the inward unit normal $g$ of $P$
satisfies the inequality $\| f - g \| \le \varepsilon$.  Finally,
$K \cap P = B \cap P = \{ x \}$ implies that $K \cap P$ is a
singleton (that is, $x$ is an exposed point of $K$).
\end{proof}

\begin{lemma} \label{lemma-22}
Let $K \subset \R^n$ be a compact convex set with more than one
point and $f$ be a unit vector in $\R^n$. For any $\varepsilon
> 0$, there is a unit vector $g \in \R^n$ and opposite
closed halfspaces $P$ and $Q$ both orthogonal to $g$ and
supporting $K$ such that $\| f - g \| \le \varepsilon$ and the
sets $K \cap S$ and $K \cap T$ are distinct singletons.
\end{lemma}

\begin{proof}
Consider the compact convex set $K^* = K + (-K)$. By
Lemma~\ref{lemma-21}, there is a closed halfspace $P \subset \R^n$
such that $K^* \cap P$ is a singleton and the inward unit normal
$g$ of $P$ satisfies the inequality $\| f - g \| \le \varepsilon$.
Furthermore, $K^* \cap P \ne K^*$ since $K$ has more than one
point. Denote by $S$ and $-T$ the closed halfspaces that are
translates of $P$ and support the sets $K$ and $-K$, respectively.
From
\[
K^* \cap P = (K + (-K)) \cap P = K \cap S + (-K) \cap (-T)
\]
we conclude that both sets $K \cap S$ and $(-K) \cap (-T)$ are
singletons. Finally, $K \cap S$ and $K \cap T$ are distinct due to
$K^* \cap P \ne K^*$.
\end{proof}

We start the proof of Theorem~\ref{th2} by considering the set $E$
of antipodally exposed points of $K$. Obviously, $\Cl (\Conv E)
\subset K$; so it remains to show the opposite inclusion. Assume,
for contradiction, the existence of a point $a \in K \setminus \Cl
(\Conv E)$. By the separation properties of convex sets, there is
a closed halfspace $Q \subset \R^n$ that contains $a$ and is
disjoint from $\Cl (\Conv E)$. Denote by $Q'$ the translate of $Q$
that supports $K$. Clearly, $Q' \subset Q$; so $Q' \cap \Cl (\Conv
E) = \varnothing$. We can write $Q' = \{ x \in \R^n \mid x \!
\cdot \! f \ge \gamma \}$, where $f$ is the inward unit normal of
$Q'$ and $\gamma$ is a suitable scalar. Since the set $\Cl (\Conv
E)$ is compact, there is an $\varepsilon > 0$ such that any closed
halfspace $P \subset \R^n$ with inward unit normal $e$ is disjoint
from $\Cl (\Conv E)$ provided $P$ supports $K$ and $\| f - e \|
\le \varepsilon$. By Lemma~\ref{lemma-22}, there is a unit vector
$g$ with $\| f - g \| \le \varepsilon$ and a pair of opposite
closed halfspaces $S$ and $T$ of the form (\ref{eqn1}) such that
$K \cap S$ and $K \cap T$ are distinct singletons. If $K \cap S =
\{ u \}$ and $K \cap T = \{v\}$, then $u$ and $v$ are antipodally
exposed points of $K$. Finally, $S \cap (\Cl \Conv E) =
\varnothing$ implies $u \notin \Cl (\Conv E)$, a contradiction.
\qed

\section{Proof of Theorem~\ref{th1}}

\noindent Obviously, $1) \Rightarrow 2)$. We start the proof of
the converse statement by considering the case when both $K_1$ and
$K_2$ are compact.

\medskip

\noindent \textbf{Case I.} Both $K_1$ and $K_2$ are compact and $2
\le m \le n - 1$.

\medskip

Since 2) trivially implies 1) when both $K_1$ and $K_2$ are
singletons, we may assume, in what follows, that each of $K_1$ and
$K_2$ has more than one point.

\medskip

\noindent \textbf{A)} We consider the case $m = n - 1$ separately,
dividing our consideration into a sequence of steps.

\medskip

\noindent \textbf{1.} First, we state that for any exposed
diameter $[x_1, z_1]$ of $K_1$ and opposite closed halfspaces
$P_1$ and $Q_1$ of $\R^n$ with the property
\[
K_1 \cap P_1 = \{ x_1 \} \quad \mathrm{and} \quad K_1 \cap Q_1 =
\{ z_1 \},
\]
there is an exposed diameter $[x_2, z_2]$ of $K_2$ parallel to
$[x_1, z_1]$ and opposite closed halfspaces $P_2$ and $Q_2$ of
$\R^n$ that are translates of $P_1$ and $Q_1$, respectively, such
that
\[
K_2 \cap P_2 = \{ x_2 \} \quad \mathrm{and} \quad K_2 \cap Q_2 =
\{ z_2 \}.
\]

Indeed, denote by $P_2$ and $Q_2$ some translates of $P_1$ and
$Q_1$, respectively, that support $K_2$. Clearly, $P_2 \cap Q_2 =
\varnothing$. Choose any points $x_2 \in K_2 \cap P_2$ and $z_2
\in K_2 \cap Q_2$. Assume for a moment that $[x_2, z_2]$ is not
parallel to $[x_1, z_1]$. Then the line through $x_1$ parallel to
$[x_2, z_2]$ intersects the hyperplane $\Bd Q_1$ at a point $z_1'$
distinct from $z_1$. Choose in $\Bd Q_1$ a line $l$ through $z_1$
orthogonal to the line $(z_1, z_1')$ and denote by $L$ the
hyperplane through $z_1$ orthogonal to $l$. Clearly, the parallel
$(n - 2)$-dimensional planes $L \cap \Bd P_i$ and $L \cap \Bd Q_i$
are distinct and support the orthogonal projection $\pi_L(K_i)$,
$i = 1, 2$, such that
\begin{gather*}
(L \cap \Bd P_1) \cap \pi_L(K_1) = \{ \pi_L(x_1) \}, \quad  (L
\cap \Bd Q_i) \cap \pi_L(K_1) = \{ \pi_L(z_1)
\}, \\
\pi_L(x_2) \in (L \cap \Bd P_2) \cap \pi_L(K_2), \quad  \pi_L(z_2)
\in (L \cap \Bd Q_2) \cap \pi_L(K_2).
\end{gather*}
By the hypothesis, $\pi_L(K_1)$ and $\pi_L(K_2)$ are homothetic.
Hence there is an exposed diameter $[u,v]$ of $\pi_L(K_2)$
parallel to $[\pi_L(x_1), \pi_L(z_1)]$ such that
\[
(L \cap \Bd P_2) \cap \pi_L(K_2) = \{ u \}, \quad (L \cap \Bd Q_2)
\cap \pi_L(K_2) = \{ v \}.
\]
This gives $\pi_L(x_2) = u$ and $\pi_L(z_2) = v$, which is
impossible because the line segments $[\pi_L(x_1), \pi_L(z_1)]$
and $[\pi_L(x_2), \pi_L(z_2)]$ are not parallel. The obtained
contradiction shows that $[x_2, z_2]$ is parallel to $[x_1, z_1]$
for any choice of $x_2 \in K_2 \cap P_2$ and $z_2 \in K_2 \cap
Q_2$. Hence both sets $K_2 \cap P_2$ and $K_2 \cap Q_2$ are
singletons, which implies that $[x_2, z_2]$ is an exposed diameter
of $K_2$ parallel to $[x_1, z_1]$.

\medskip

\noindent \textbf{2.} Choose an exposed diameter $[x_0, z_0]$ of
$K_1$ and denote by $[x_0', z_0']$ the exposed diameter of $K_2$
parallel to $[x_0, z_0]$ (the uniqueness of $[x_0', z_0']$ follows
from Lemma~\ref{lemma-a1}). Replacing $K_1$ with $K_1 - (x_0 +
z_0)/2$ and $K_2$ with
\[
\lambda (K_2 - (x_0' + z_0')/2), \quad \lambda = \|x_0 - z_0\| /
\|x_0' - z_0'\|,
\]
we may assume that $[x_0, z_0]$ is an exposed diameter for both
$K_1$ and $K_2$, centered at $o$. By \textbf{1} above, both $K_1$
and $K_2$ are supported by opposite closed halfspaces $P_0$ and
$Q_0$ such that
\[
K_1 \cap P_0 = K_2 \cap P_0 = \{x_0\}, \quad K_1 \cap Q_0 = K_2
\cap Q_0 = \{z_0\}.
\]
Applying, if necessary, a suitable affine transformation, we may
assume that both hyperplanes $\Bd P_0$ and $\Bd Q_0$ are
orthogonal to $[x_0, z_0]$. Clearly, the orthogonal projections of
the transformed sets $K_1$ and $K_2$ on any plane are homothetic.

\medskip

\noindent \textbf{3.} We state that any exposed diameter $[x_2,
z_2]$ of $K_2$ is a translate of a suitable exposed diameter
$[x_1, z_1]$ of $K_1$.

\medskip

Since this statement trivially holds when $[x_2, z_2] = [x_0,
z_0]$, we assume, in what follows, that $[x_2, z_2] \ne [x_0,
z_0]$. Let $P_2$ and $Q_2$ be opposite closed halfspaces of $\R^n$
with the property $K_2 \cap P_2 = \{x_2\}$ and $K_2 \cap Q_2 =
\{z_2\}$. Denote by $P_1$ and $Q_1$ translates of $P_2$ and $Q_2$,
respectively, that support $K_1$. By \textbf{1} above, the sets
$K_1 \cap P_1$ and $K_1 \cap Q_1$ are singletons, say, $\{x_1\}$
and $\{z_1\}$, such that $[x_1,z_1]$ and $[x_2,z_2]$ are parallel.
Clearly, $P_1 \ne P_0 \ne P_2$ and $Q_1 \ne Q_0 \ne Q_2$ due to
$[x_2, z_2] \ne [x_0, z_0]$.

Choose a line $l \subset \Bd P_0 \cap \Bd P_1$ and denote by $L$
the hyperplane through $[x_0, z_0]$ orthogonal to $l$. Clearly,
$\pi_L(K_i)$, $i = 1, 2$, is a compact convex set distinct from a
singleton and bounded by two pairs of parallel $(n -
2)$-dimensional planes
\[
L \cap \Bd P_0, \ L \cap \Bd Q_0 \quad \textrm{and} \quad L \cap
\Bd P_i, \ L \cap \Bd Q_i.
\]
This shows that both $[\pi_L(x_0), \pi_L(z_0)]$ and $[\pi_L(x_i),
\pi_L(z_i)]$ are exposed diameters of $\pi_L(K_i)$, $i = 1, 2$.
Since $\pi_L(K_1)$ and $\pi_L(K_2)$ are homothetic and share an
exposed diameter $[\pi_L(x_0), \pi_L(z_0)]$, the set $\pi_L(K_2)$
equals one of the sets $\pi_L(K_1)$, $\pi_L(-K_1)$. In either
case, $[\pi_L(x_2), \pi_L(z_2)]$ is a translate of $[\pi_L(x_1),
\pi_L(z_1)]$. Because $[x_1, z_1]$ and $[x_2, z_2]$ are parallel,
we conclude that $[x_2, z_2]$ is a translate of $[x_1, z_1]$.

\medskip

\noindent \textbf{4.} Our next statement (in continuation of
\textbf{3} above) is that the exposed diameter $[x_2, z_2]$ of
$K_2$ coincides with $[x_1, z_1]$ or with $[-x_1, -z_1]$.

\medskip

Indeed, by the proved in \textbf{3} above, $\pi_L(K_2)$ equals one
of the sets $\pi_L(K_1)$, $\pi_L(-K_1)$; whence its exposed
diameter $[\pi_L(x_2), \pi_L(z_2)]$ coincides with one of the line
segments $[\pi_L(x_1), \pi_L(z_1)]$, $[\pi_L(-x_1), \pi_L(-z_1)]$.
Without loss of generality, we may assume that
\begin{equation} \label{segm1}
[\pi_L(x_2), \pi_L(z_2)] = [\pi_L(x_1), \pi_L(z_1)].
\end{equation}

Let $M$ be the hyperplane through $[x_0, z_0]$ parallel to the $(n
- 2)$-dimen\-sio\-nal plane $\Bd P_0 \cap \Bd P_1$. Denote by $M'$
a hyperplane (distinct from both $L$ and $M$) that contains the
$(n - 2)$-dimensional plane $L \cap M$, and let $P_i'$ and $Q_i'$
be the opposite closed halfspaces of $\R^n$ both supporting $K_i$
whose boundary hyperplanes $\Bd P_i'$ and $\Bd Q_i'$ are parallel
to $M'$, $i = 1, 2$. Consider the hyperplane $L'$ through $L \cap
M$ that forms an angle of $90^\circ$ with $M'$. If $\pi_L'$ is the
orthogonal projection of $\R^n$ onto $L'$, then the homothetic set
$\pi_L'(K_1)$ and $\pi_L'(K_2)$ have $[x_0, z_0]$ as a common
exposed diameter, which implies that $\pi_L'(K_2)= \pi_L'(K_1)$ or
$\pi_L'(K_2) = \pi_L'(-K_1)$. Clearly, the equality $\pi_L'(K_2) =
\pi_L'(K_1)$ gives $P_2' = P_1'$, and the equality $\pi_L'(K_2) =
\pi_L'(-K_1)$ gives $P_2' = -Q_1'$.

Assume, for contradiction, that
\[
[x_1, z_1] \ne [x_2, z_2] \ne [-x_1, -z_1].
\]
Due to (\ref{segm1}), both lines $(x_1, x_2)$ and $(z_1, z_2)$ are
parallel to $l$. Because $x_2$ and $z_2$ are the only points of
contact of $K_2$ with $P_2$ and $Q_2$, respectively, there is an
$\varepsilon_1 > 0$ so small that if the angle $\gamma$ between
$M$ and $M'$ is positive and less than $\varepsilon_1$, then
either $x_1 \in \Int P_2'$ or $z_1 \in \Int Q_2'$. In either case,
$P_1' \ne P_2'$ for all $\gamma \in \, ]0, \varepsilon_1[$; whence
$\pi_L'(K_2) \ne \pi_L'(K_1)$ for all $\gamma \in \, ]0,
\varepsilon_1[$.

Under assumption (\ref{segm1}), we consider two more subcases.

\medskip

\noindent \textbf{4a.} \ $[\pi_L(x_2), \pi_L(z_2)] = [\pi_L(-z_1),
\pi_L(-x_1)]$ \ (see part $(i)$ of Figure~\ref{fig-2}).

\medskip

Then $-z_1 \in (x_1, x_2)$ and $-x_1 \in (z_1, z_2)$. As above,
there is a scalar $\varepsilon_2 > 0$ so small that if the angle
$\gamma$ between $M$ and $M'$ is positive and less than
$\varepsilon_2$, then either $-z_1 \in \Int P_2'$ or $-x_1 \in
\Int Q_2'$. In either case, $P_2' \ne -Q_1'$ for all $\gamma \in
\, ]0, \varepsilon_2[$; whence $\pi_L'(K_2) \ne \pi_L'(-K_1)$ for
all $\gamma \in \, ]0, \varepsilon_2[$.

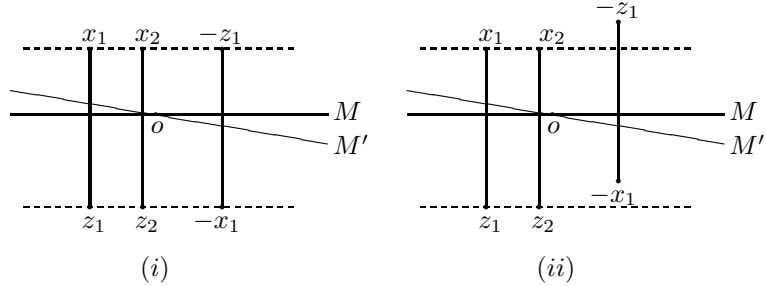
\begin{figure}[h]
\begin{minipage}{11.7cm}

\begin{center}

%\fbox{
\begin{picture}(300,105)

\put(25,60){\line(1,0){120}} \put(147,58){$M$} \put(147,45){$M'$}
\put(80,60){\circle{1}} \put(25,69){\line(6,-1){120}}
\put(78,53){$o$}

\put(175,60){\line(1,0){120}} \put(297,58){$M$} \put(297,45){$M'$}
\put(230,60){\circle{1}} \put(175,69){\line(6,-1){120}}
\put(228,53){$o$}

\put(74,-1){$(i)$} \put(224,-1){$(ii)$}

\multiput(30,25)(4,0){26}{\line(1,0){2}}
\multiput(30,85)(4,0){26}{\line(1,0){2}}

\multiput(180,25)(4,0){26}{\line(1,0){2}}
\multiput(180,85)(4,0){26}{\line(1,0){2}}

\thicklines

\put(55,85){\line(0,-1){60}} \put(55,85){\circle{1}}
\put(55,25){\circle{1}}  \put(52,88){$x_1$} \put(52,17){$z_1$}

\put(75,85){\line(0,-1){60}} \put(75,85){\circle{1}}
\put(75,25){\circle{1}}  \put(72,88){$x_2$} \put(72,17){$z_2$}

\put(105,85){\line(0,-1){60}} \put(105,85){\circle{1}}
\put(105,25){\circle{1}}  \put(96,88){$-z_1$} \put(94,17){$-x_1$}

\put(205,85){\line(0,-1){60}} \put(205,85){\circle{1}}
\put(205,25){\circle{1}}  \put(202,88){$x_1$} \put(202,17){$z_1$}

\put(225,85){\line(0,-1){60}} \put(225,85){\circle{1}}
\put(225,25){\circle{1}}  \put(225,88){$x_2$} \put(222,17){$z_2$}

\put(255,95){\line(0,-1){60}} \put(255,95){\circle{1}}
\put(255,35){\circle{1}}  \put(246,98){$-z_1$}
\put(244,27){$-x_1$}

\end{picture}
%}

\end{center}

\end{minipage}
\caption{Illustration of subcases 4a and 4b.} \label{fig-2}
\end{figure}

\noindent \textbf{4b.} \ $[\pi_L(x_2), \pi_L(z_2)] \ne
[\pi_L(-z_1), \pi_L(-x_1)]$ \ (see part $(ii)$ of
Figure~\ref{fig-2}).

\medskip

In particular, $\pi_L(x_2) \ne \pi_L(-z_1)$. Because of $K_2 \cap
P_2 = \{ x_2 \}$ and $(-K_1) \cap (-Q_1) = \{ -z_1 \}$, there is a
scalar $\varepsilon_3 > 0$ such that if the angle $\gamma$ between
$M$ and $M'$ is positive and less than $\varepsilon_3$, then the
compact sets $K_2 \cap P_2'$ and $(-K_1) \cap (-Q_1')$ are small
enough: that is, for any points $u \in K_2 \cap P_2'$ and $v \in
(-K_1) \cap (-Q_1')$,
\begin{equation} \label{4b-1}
\begin{aligned}
\|u - x_2 \| & \le \tfrac{1}{4} \|\pi_L(x_2) - \pi_L(-z_1) \|, \\
\|v - (-z_1) \| &\le \tfrac{1}{4} \|\pi_L(x_2) - \pi_L(-z_1) \|.
\end{aligned}
\end{equation}
By continuity, $\varepsilon_3$ can be chosen so small that
\begin{equation} \label{4b-2}
\|\pi_L'(x_2) - \pi_L'(-z_1) \| \ge  \tfrac{3}{4} \|\pi_L(x_2) -
\pi_L(-z_1)\|
\end{equation}
for all $\gamma \in \, ]0, \varepsilon_3[$. Together with
\[
\| \pi_L'(u) - \pi_L'(x_2) \| \le \| u - x_2 \|, \quad \|
\pi_L'(v) - \pi_L'(-z_1) \| \le \| v - (- z_1) \|,
\]
the inequalities (\ref{4b-1}) and (\ref{4b-2}) give
\begin{equation} \label{4b-3}
\begin{aligned}
&\| \pi_L'(u) - \pi_L'(v) \| \\
& \ge \| \pi_L'(x_2) - \pi_L'(-z_1) \| - \|\pi_L'(u) - \pi_L'(x_2)
\| - \|\pi_L'(v) - \pi_L'(-z_1) \| \\
& \ge \tfrac{3}{4} \| \pi_L(x_2) - \pi_L(-z_1) \| - \| u - x_2 \|
- \| v - (-z_1) \| \\
& \ge \tfrac{1}{4} \| \pi_L(x_2) - \pi_L(-z_1) \|.
\end{aligned}
\end{equation}
Since $\pi_L'(K_2)$ is supported by $P_2'$ and $\pi_L'(-K_1)$ is
supported by $-Q_1'$, which is a translate of $P_2'$, the
inequality (\ref{4b-3}) shows that the contact sets
\begin{gather*}
\pi_L'(K_2) \cap P_2' = \pi_L'(K_2 \cap P_2'), \\
\pi_L'(-K_1) \cap (-Q_1') = \pi_L'((-K_1) \cap (-Q_1'))
\end{gather*}
are disjoint for all $\gamma \in \, ]0, \varepsilon_3[$. Hence
$\pi_L'(K_2) \ne \pi_L'(-K_1)$ for all $\gamma \in \, ]0,
\varepsilon_3[$.

Finally, with $\varepsilon_0 = \min \{ \varepsilon_1,
\varepsilon_2, \varepsilon_3 \}$, we have
\[
\pi_L'(K_1) \ne \pi_L'(K_2) \ne \pi_L'(-K_1) \ \  \mbox{for all} \
\ \gamma \in \, ]0, \varepsilon_0[,
\]
in contradiction with the condition that $\pi_L'(K_2)$ equals one
of the sets $\pi_L'(K_1)$, $\pi_L'(-K_1)$. Thus $[x_2, z_2]$
coincides with $[x_1, z_1]$ or with $[-x_1, -z_1]$.

\medskip

\noindent \textbf{5.} Our concluding statement (in continuation of
\textbf{4}) is that $K_2 = K_1$ or $K_2 = -K_1$.

\medskip

Indeed, assume for a moment that $K_1 \ne K_2 \ne -K_1$. Since
$K_1 \ne K_2$, Theorem~\ref{th2} implies that $K_1$ has an exposed
diameter $[u_1, v_1]$ that is not an exposed diameter of $K_2$.
Then \textbf{4} above implies that $[-v_1, -u_1]$ is a common
exposed diameter of $K_2$ and $-K_1$. In particular, $[u_1, v_1]
\ne [-v_1, -u_1]$. Similarly, $K_2 \ne -K_1$ implies the existence
of an exposed diameter $[-v_0, -u_0]$ of $-K_1$ which is not an
exposed diameter of $K_2$, while $[u_0, v_0]$ is a common exposed
diameter of $K_1$ and $K_2$. Again, $[u_0, v_0] \ne [-v_0, -u_0]$.
By Lemma~\ref{lemma-a1}, $[u_0, v_0]$ and $[u_1, v_1]$ are not
parallel.

Denote by $w$ the middle point of $[u_0, v_0]$ and consider the
sets $K_1' = K_1 - w$ and $K_2' = K_2 - w$. We observe that $w \ne
o$ because of $[u_0, v_0] \ne [-v_0, -u_0]$. The origin $o$ is the
middle point of the exposed diameter $[u_0 - w, v_0 - w]$ of
$K_1'$, which is also an exposed diameter of $K_2'$. By \textbf{4}
above (with $[u_0, v_0]$ instead of $[x_0, z_0]$), we see that
every exposed diameter of $K_2'$ is an exposed diameter of $K_1'$
or $-K_1'$. In particular, the exposed diameter $[-v_1 - w, -u_1 -
w]$ of $K_2'$ should coincide either with the exposed diameter
$[u_1 - w, v_1 - w]$ of $K_1'$ or with the exposed diameter $[-v_1
+ w, -u_1 + w]$ of $-K_1'$.

On the other hand,
\[
[-v_1 - w, -u_1 - w] \ne [u_1 - w, v_1 - w]
\]
due to $[-v_1, -u_1] \ne [u_1, v_1]$, and
\[
[-v_1 - w, -u_1 - w] \ne [-v_1 + w, -u_1 + w]
\]
because of $w \ne o$. The obtained contradiction shows that $K_2 =
K_1$ or $K_2 = -K_1$, which concludes the proof of Case I for $m =
n - 1$.

\medskip

\noindent \textbf{B)} \ Now we assume that $2 \le m < n - 1$. Let
$M \subset \R^n$ be a plane of dimension $m + 1$. For any plane $L
\subset M$ of dimension $m$, we can express $\pi_L$ as the
composition $\pi_L = \pi' \circ \pi_M$, where $\pi'$ is the
orthogonal projection of $M$ onto $L$. This observation and
condition 2) of the theorem imply that the orthogonal projections
of the sets $\pi_M(K_1)$ and $\pi_M(K_2)$ on every $m$-dimensional
plane $L \subset M$ are homothetic. By the proved above (with $m +
1$ instead of $n$), the sets $\pi_M(K_1)$ and $\pi_M(K_2)$ are
homothetic. Since this argument holds for every $(m +
1)$-dimensional plane in $\R^n$, we can replace $m$ with $m + 1$
in condition 2) of the theorem. Repeating this argument finitely
many times, we see that the orthogonal projections of $K_1$ and
$K_2$ on each hyperplane of $\R^n$ are homothetic. By the proved
above, $K_1$ and $K_2$ are homothetic themselves.

\medskip

\noindent \textbf{Case II.} At least one of the sets $K_1$ and
$K_2$ is unbounded and $3 \le m \le n - 1$.

\medskip

Let, for example, $K_1$ be unbounded. Then $\Rec K_1 \ne \{o\}$.
Choose a closed halfline $h$ with apex $o$ that lies in $\Rec K_1$
and an $m$-dimensional subspace $L$ that contains $h$. Then $h
\subset \Rec \pi_L(K_1)$, which shows that $\pi_L(K_1)$ is
unbounded. Since $\pi_L(K_2)$ is homothetic to $\pi_L(K_1)$, the
set $\pi(K_2)$ is also unbounded, which implies that $K_2$ is
unbounded.

\medskip

\noindent \textbf{6.} We state that $\Lin K_1 = \Lin K_2$.

\medskip

Indeed, assume, for example, that $\Lin K_1$ contains a line $l$
through $o$ that does not belong to $\Lin K_2$. Then $l$ does not
lie entirely in $\Rec K_2$, since otherwise $l$ would belong to
$\Lin K_2$. Let $h$ be a halfline of $l$ with apex $o$ that does
not lie in $\Rec K_2$. Because $\Rec K_2$ is a closed convex cone
with apex $o$, there is a closed halfspace $Q$ that contains $\Rec
K_2$ and is disjoint from $h \setminus \{ o \}$. Clearly, $o \in
\Bd Q$. Choose an $(n - m)$-dimensional subspace $N$ in $\Bd Q$,
and denote by $L$ the orthogonal complement to $N$. Clearly, the
line $\pi_L(l)$ lies in $\Rec \pi_L(K_1)$ and does not lie in
$\Rec \pi_L(K_2)$, which belongs to $L \cap Q$. The last is
impossible because $\pi_L(K_1)$ and $\pi_L(K_2)$ are homothetic by
condition 2). Hence $\Lin K_1 \subset \Lin K_2$. Similarly, $\Lin
K_2 \subset \Lin K_1$.

\medskip

\noindent \textbf{7.} Due to \textbf{6} above, both $K_1$ and
$K_2$ can be expressed as
\begin{equation} \label{comp}
K_1 = \Lin K_1 \oplus (K_1 \cap M), \quad K_2 = \Lin K_1 \oplus
(K_2 \cap M),
\end{equation}
where the subspace $M$ is the orthogonal complement of $\Lin K_1$
and both sets $K_1 \cap M$ and $K_2 \cap M$ are line-free.

First assume that $\Dim M \le m$. In this case, we choose an
$m$-dimensional subspace $L \subset \R^n$ that contains $M$.
Clearly,
\[
\pi_L(K_i) = (\Lin K_1 \cap L) \oplus (K_i \cap M), \quad i = 1,
2.
\]
Then $K_1 \cap M$ and $K_2 \cap M$ are homothetic because the sets
$\pi_L(K_1)$ and $\pi_L(K_2)$ are homothetic by the hypothesis.
This and (\ref{comp}) imply that $K_1$ and $K_2$ are homothetic
themselves.

Now assume that $\Dim M > m$. Since $K_1 \cap M$ is line-free, it
contains an exposed point $x$. Translating $K_1$ on the vector
$-x$, we may assume that $o$ is an exposed point of $K_1$. Let $N$
be a subspace of $M$ of dimension $\Dim M - 1$ that supports $K_1
\cap M$ such that $N \cap (K_1 \cap M) = \{o\}$.  Denote by $N_+$
and $N_-$ the opposite closed halfplanes of $M$ bounded by $N$.
Let, for example, $K_1 \cap M \subset N_+$. Denote by $l$ the
1-dimensional subspace of $M$ orthogonal to $N$. Choose an
$m$-dimensional subspace $S$ of $M$ that contains $l$. By the
above, $\pi_S(K_1 \cap M) \subset S \cap N_+$.

\medskip

\noindent \textbf{7a.} If $\pi_S(K_2 \cap M)$ is positively
homothetic to $\pi_S(K_1 \cap M)$, then the recession cones of
$\pi_S(K_1 \cap M)$ and $\pi_S(K_2 \cap M)$ coincide. This shows
that for any other $m$-dimensional subspace $S'$ of $M$ that
contains $l$, the orthogonal projections of $K_1 \cap M$ and $K_2
\cap M$ on $S'$ are positively homothetic. Since $m \ge 3$, it
follows from~\cite{sol05b} that $K_1 \cap M$ and $K_2 \cap M$ are
positively homothetic, and (\ref{comp}) implies that $K_1$ and
$K_2$ are positively homothetic themselves.

\medskip

\noindent \textbf{7b.} If $\pi_S(K_2 \cap M)$ is negatively
homothetic to $\pi_S(K_1 \cap M)$, then the recession cones of
$\pi_S(K_1 \cap M)$ and $\pi_S(K_2 \cap M)$ are symmetric about
$o$. This shows that for any other $m$-dimensional subspaces $S'$
of $M$ that contains $l$, the orthogonal projections of $K_1 \cap
M$ and $K_2 \cap M$ on $S'$ are negatively homothetic. Since $m
\ge 3$, it follows from~\cite{sol05b} that $K_1 \cap M$ and $K_2
\cap M$ are negatively homothetic, and (\ref{comp}) implies that
$K_1$ and $K_2$ are negatively homothetic themselves. \qed

\section{Proof of Corollary~\ref{cor1}}

\noindent Because 1) obviously implies 2), it remains to show that
$2) \Rightarrow 1)$. Let compact convex sets $K_1$ and $K_2$ in
$\R^n$ satisfy condition 2) of the corollary. Choose any
2-dimensional subspace $L \subset \R^n$. Since $\Dim (L + S) \le r
+ 2 \le m$, there is an $m$-dimensional subspace $M$ that contains
$L + S$. By condition 2), $\pi_M(K_1)$ and $\pi_M(K_2)$ are
homothetic. This implies that the orthogonal projections of the
sets $\pi_M(K_1)$ and $\pi_M(K_2)$ onto $L$ are homothetic.
Because $\pi_L = \pi' \circ \pi_M$, where $\pi'$ is the orthogonal
projection of $M$ onto $L$, we conclude that $\pi_L(K_1)$ and
$\pi_L(K_2)$ are homothetic. Now Theorem~\ref{th1} (with $m = 2$)
implies that $K_1$ and $K_2$ are homothetic themselves.

If $K_1$ and $K_2$ are closed convex sets that satisfy condition
2) of the corollary, then repeating the argument above, with any
3-dimensional subspace $L \subset \R^n$ and the respective
inequality $\Dim (L + S) \le r + 3 \le m$, we obtain the homothety
of $K_1$ and $K_2$. \qed

%\textbf{Acknowledgement.} The author would like to thank
%the referee for many helpful comments on an earlier draft of this
%paper.

%\normalsize\smallskip
%\renewcommand*{\refname}{\large\textbf{References}}

\end{document}